\documentclass{gtpart}
\usepackage{pinlabel}
\usepackage[all]{xy}
\usepackage{amscd}
\title{On the cobordism groups of cooriented, codimension one Morin maps}

\author{Andr\'as Sz\H{u}cs}
\givenname{Andr\'as}
\surname{Sz\H{u}cs}
\address{Department of Analysis\\
E\"otv\"os University\\\newline
P\'azm\'any P\'eter s\'et\'any 1/C\\
H--1117 Budapest, Hungary}
\email{szucs@cs.elte.hu}
\urladdr{}

\arxivreference{}  %%% please supply if
\arxivpassword{}   %%% paper is in the arXiv

\newtheorem{thm}{Theorem} %%[section]
\newtheorem{lem}[thm]{Lemma}
\newtheorem{claimnumbered}{Claim}
\newtheorem*{theor}{Theorem}
\newtheorem*{claim}{Claim}
\newtheorem*{claima}{Claim a)}
\newtheorem*{claimb}{Claim b)}

\theoremstyle{definition}
\newtheorem*{defn}{Definition}
\newtheorem*{rem}{Remark}

\makeop{Homo}
\numberwithin{section}{part}

\def\co{\colon\thinspace}
\def\Cob{\text{\rm Cob}}
\def\Imm{\text{\rm Imm}}
\def\Ker{\text{\rm Ker}}
\def\lra{\longrightarrow }
\def\wt{\widetilde }
\def\wh{\widehat }
\def\fr{\text{\rm fr}}
\def\circledone{\textstyle\bigcirc\kern-8.55pt\raisebox{-.3pt}{$\scriptstyle 1$}\hspace*{2pt}}
\def\circledtwo{\textstyle\bigcirc\kern-8.4pt\raisebox{-.3pt}{$\scriptstyle 2$}\hspace*{2pt}}
\def\circledthree{\textstyle\bigcirc\kern-8.4pt\raisebox{-.3pt}{$\scriptstyle 3$}\hspace*{2pt}}
\begin{document}
\begin{abstract}    % type your abstract below

\end{abstract}
\maketitle

\part{Cobordism of fold maps and the Kahn--Priddy map}
\label{part:I}

\section{Formulation of the result}
\label{sec:1}

Let us denote by $\text{\rm Cob}\, \Sigma^{1,0}(n)$ the cobordism group of cooriented, codimension~$1$ fold maps of closed, smooth, $n$--dimensional manifolds in $R^{n + 1}$ (see \cite{Sz3}).

(A fold map may have only $\Sigma^{1,0}$--type (or $A_1$--) singular points, see \cite{AGV}.)

\renewcommand{\thethm}{\Alph{thm}}
\begin{thm}
\label{th:A}
\begin{itemize}
\item[\emph{\rm a)}]
$\Cob\, \Sigma^{1,0}(n)$ is a finite Abelian group.
\item[{\rm b)}]
Its odd torsion part is isomorphic to that of the $n$\textsuperscript{th}  stable homotopy group of spheres, i.e.\ for any odd prime $p$ $\Cob\, \Sigma^{1,0}(n)_p \approx \pi^s(n)_p$, where the lower index $p$ denotes the $p$--primary part.
\item[{\rm c)}]
Its $2$--primary part is isomorphic to the kernel of the Kahn--Priddy homomorphism \cite{KP}:
\[
\lambda_* \co \pi_{n - 1}^s (RP^\infty) \lra \pi^s(n - 1).
\]
\end{itemize}
\end{thm}

\renewcommand{\thethm}{\arabic{thm}}
\setcounter{thm}{0}
\begin{rem}
The group $\pi_{n - 1}^s(RP^\infty)$ is a $2$--primary group and $\lambda_*$ is onto the $2$--primary part of $\pi^s(n - 1)$, see \cite{KP}.
\end{rem}

Hence: $\Cob\, \Sigma^{1,0}(n) \approx \pi^s(n)_{\{\text{odd torsion part}\}} \oplus \Ker \left(\lambda^* \co \pi_{n - 1}^s (RP^\infty) \lra \pi^s(n - 1)\right)$.

\section{The Kahn--Priddy map (\cite{KP}, \cite{K})}
\label{sec:2}

Let us consider the composition of the following maps:

a) $RP^{q - 1} \hookrightarrow O(q)$.
A line $L \subset R^q$, $[L] \in RP^{q - 1}$ is mapped into the reflection in its orthogonal hyperplane.

b) $O(q) \hookrightarrow \Omega^q S^q$ maps $A \in O(q)$ to the map
\begin{eqnarray*}
S^q = R^2 \cup \infty & \lra  & S^q = R^q \cup \infty \, \,\text{ defined as}\\
x\qquad & \lra  & \quad A(x) \qquad \quad \text{ for } \ x \in R^q \ \text{ and}\\
\infty \quad\ \ & \lra  & \quad\ \infty.
\end{eqnarray*}

Take the adjoint of the composition map $RP^{q - 1} \lra \Omega^q S^q$.
It is a map $\lambda \co \Sigma^q RP^{q - 1} \lra  \Omega^q S^q$.

If $n < q$, then the homotopy groups
$\pi_{q + n} (\Sigma^q \, RP^{q - 1})$ and $\pi_{q + n}(S^q)$ are stable, and
\[
\pi_n^s(RP^{q - 1}) \approx \pi_n^s(RP^\infty).
\tag{$*$}
\]
The Kahn--Priddy homomorphism $\lambda_*\co \pi_n^s(RP^\infty) \lra \pi^s(n)$ is the homomorphism induced by $\lambda$ in the stable homotopy groups (precomposed with the isomorphism $(*)$).

\begin{thm}[Kahn--Priddy \cite{KP}]
\label{th:1}
$\lambda_*$ is onto the $2$--primary component of $\pi^s(n)$.
\end{thm}

\section{Koschorke's interpretation of {$\lambda_*$}}

Ulrich Koschorke gave a very geometric description of the Kahn--Priddy homomorphism through the so-called ``figure 8 construction''.
Given an immersion of an $(n - 1)$--dimensional (unoriented) manifold $N^{n - 1}$ into $R^n$ the figure 8 construction associates with it an immersion of an oriented $n$--dimensional manifold $M^n$ into $R^{n + 1}$ as follows:

Let us consider the composition $N^{n - 1} \looparrowright R^n \hookrightarrow R^{n + 1}$.

This has normal bundle of the form $\varepsilon^1 \oplus \zeta^1$, where $\varepsilon^1$ is the trivial line bundle (the $(n + 1)$\textsuperscript{th}  coordinate direction in $R^{n + 1}$) and $\zeta^1$ is the normal line bundle of $N^{n - 1}$ in $R^n$.

Let us put a figure 8 in each fiber of $\varepsilon^1 \oplus \zeta^1$ symmetrically with respect to the reflection in the fiber~$\zeta^1$.
Choosing these figures 8 smoothly their union gives the image of an immersion of an oriented $n$--dimensional manifold $M^n$ into $R^{n + 1}$.
(Clearly $M^n$ is the total space of the circle bundle $S(\varepsilon^1 \oplus \zeta^1)$ over $N^{n - 1}$.)

This construction gives a map $8_*\co \pi_n^s(RP^\infty) \lra \pi_n^s(n)$.
Indeed, the cobordism group of immersion of unoriented $(n - 1)$--dimensional manifolds in $R^n$ is isomorphic to $\pi_n^s(RP^\infty)$, and that of oriented $n$--dimensional manifolds in $R^{n + 1}$ is $\pi^s(n)$.

Since the figure 8 construction respects the cobordism relation (i.e.\ it associates to cobordant immersions such ones) we obtain a map of the cobordism groups.

\begin{thm}[Koschorke, {\cite[Theorem 2.1]{K}}]
\label{th:2}
The maps $\lambda_*$ and $8_*$ coincide.
\end{thm}

This theorem of Koschorke will be the main tool in the computation of the cobordism groups of fold maps.

\section{Generalities on the cobordisms of singular maps}
\label{sec:4}

In \cite{Sz3} we considered cobordism groups of singular maps with a given set $\tau$ of allowed local forms.
(Such a map was called a $\tau$--map.)
The cobordism group of (cooriented) $\tau$--maps of $n$--dimensional manifolds in Euclidean space was denoted by $\Cob_\tau(n)$.
A classifying space $X_\tau$ has been constructed for $\tau$--maps with the property that its homotopy groups are isomorphic to the groups $\Cob_\tau(n)$.

An ancestor of the spaces $X_\tau$ was the classifying space for the cobordism groups of immersions.
Namely given a vector bundle $\xi^k$ we denote by $\Imm^\xi(n)$ the cobordism group of immersions of $n$--manifolds in $R^{n + k}$ such that the normal bundle is induced from~$\xi$.
There is a classifying space $Y(\xi)$ such that
\[
\pi_{n + k}(Y(\xi)) \approx \Imm^\xi(n).
\]
Namely $Y(\xi) = \Gamma(T\xi)$, where $T\xi$ denotes the Thom space of the bundle $\xi$, and $\Gamma = \Omega^\infty S^\infty$.
(This follows by a slight modification from \cite{W}.)

Next we recall the so-called ``key bundle'', that is the main tool in handling cobordism groups of singular maps.

Let $\tau$ be a list of allowed local forms, and let $\eta$ be a maximal element in it.
(The set of local forms has a natural partial ordering, $\eta$ is greater than $\eta'$ if an isolated $\eta$--germ (at the origin) has an $\eta'$--point arbitrarily close to the origin.)

Let $\tau'$ be $\tau\setminus \{\eta\}$ (i.e.\ we omit the maximal element~$\eta$).

Note that the stratum of $\eta$ points is immersed.
We have established in \cite{Sz3} that there is a universal bundle -- denoted by $\widetilde\xi_\eta$ -- for the normal bundles of $\eta$--strata from which these normal bundles always can be induced (with the smallest possible structure group).

In particular to the cobordism class $[f]$ of a $\tau$--map $f\co M^n\lra R^{n + k}$ we can associate the element in $\Imm^{\widetilde\xi_\eta}(m)$ represented by the restriction of~$f$ to its $\eta$--stratum.
(Here $m$ is the dimension of the $\eta$--stratum.)
Hence a homomorphism $\Cob_\tau(n) \lra \Imm^{\widetilde\xi_\eta}(m)$ arises.
Both these groups are homotopy groups (of $X_\tau$ and $\Gamma T\widetilde\xi_\eta$ respectively).
It turns out that this map is induced by a map of the classifying spaces
$X_\tau \lra \Gamma T\wt\xi_\eta$.
Moreover the latter is a Serre fibration with (homotopy) fiber $X_{\tau'}$.
(This was shown in \cite{Sz3} using some nontrivial homotopy theory.
Terpai in \cite{T} gave an elementary proof for it.
This fibration is called the ``key bundle''.)

\section{Computation of the groups {$\Cob\,\Sigma^{1,0}(n)$}}
\label{sec:5}

In the case of fold maps $\tau = \{\Sigma^0, \Sigma^{1,0}\}$ where $\Sigma^0$ denotes the germ of maximal rank and $\Sigma^{1,0}$ denotes that of a Whitney umbrella $(R^2, 0)\lra (R^3, 0)$ (multiplied by the germ of identity $(R^{n - 2}, 0) \lra (R^{n - 2}, 0)$).

Hence here $\eta = \Sigma^{1,0}$ and $\tau' = \Sigma^0$.
Note that a $\tau'$--map is nothing else but an immersion (cooriented and of codimension~$1$).
Hence $X_{\tau'} = \Gamma S^1$.
Now the key bundle looks as follows:
\[
X\Sigma^{1,0}\ \overset{\Gamma S^1}{\hbox to 12mm{\rightarrowfill}}\ \Gamma T\wt\xi_{\Sigma^{1,0}}\,.
\tag{$**$}
\]
It is not hard to see (see also \cite{RSz}) that the bundle $\wt\xi_{\Sigma^{1,0}}$ is $2\varepsilon^1 \oplus \gamma^1$, and so $T\wt\xi_{\Sigma^{1,0}} = S^2 RP^\infty$.

Now the bundle $(**)$ gives the following exact sequence of homotopy groups:
\[
\aligned
\pi_{n + 1}(\Gamma S^1)
&\lra \pi_{n + 1}(X\Sigma^{1,0}) \lra \pi_{n + 1} (\Gamma S^2 RP^\infty) \overset\partial\lra \pi_n(\Gamma S^1) \ \text{ i.e.}\\
\pi^s(n)
& \lra \Cob\, \Sigma^{1,0}(n) \lra \pi_{n - 1}^s(RP^\infty) \overset\partial\lra \pi^s(n - 1) \lra
\endaligned
\]

\begin{lem}
\label{lem:3}
The boundary map $\partial$ coincides with the map $8_*$ and hence with the Kahn--Priddy homomorphism $\lambda_*$.
\end{lem}

\begin{proof}[Proof of Theorem~\ref{th:A}]
is immediate from Theorem~\ref{th:1}, Theorem~\ref{th:2} and Lemma~\ref{lem:3}.
\end{proof}

\begin{proof}[Proof of Lemma~\ref{lem:3}]
The boundary map $\partial\co \pi_{n + 1}^s(S^2 RP^\infty) \approx \pi_{n + 1}(X\Sigma^{1,0}, \Gamma S^1) \lra \pi_n(\Gamma S^1)$
can be interpreted geometrically as follows:

The source group $\pi_{n + 1}(X\Sigma^{1,0}, \Gamma S^1)$ is isomorphic to the cobordism group of fold maps
\[
f\co (M^n, \partial M^n) \lra(D^{n + 1}, S^n) \ \text{ such that } \ f^{-1}(S^n) = \partial M^n,
\]
and the map $\partial f = f\raisebox{-3pt}{$\Big|$}_{\partial M^n}$ is an immersion of $\partial M^n$ into $S^n$.
(Here $M^n$ is an oriented compact smooth $n$--dimensional manifold with boundary $\partial M^n$.)

Let $[f]$ denote the (relative) cobordism class of~$f$, and let $[\partial f]$ be that of the immersion $\partial f\co \partial M^n \lra S^n$.
Then $\partial[f] = [\partial f]$.

Now let $V$ denote the set of singular points of~$f$.
This is a submanifold of $M^n$ of codimension~$2$.
The restriction of $f$ to $V$ is an immersion, its image we denote by $\wt V$ $(=f(V))$.
Let $\wt T$ be the (immersed) tubular neighbourhood of~$\wt V$.
More precisely there exist a $D^3$--bundle $T' \lra V$ over $V$, a submersion $F$ of $T'$ into $D^n$, $(F(T') = \wt T)$ and $F$ extends the immersion $f \mid V\co V \lra D^n$.
The bundle $T' \lra V$ has the form $2\varepsilon^1 \oplus \zeta^1$, where $\zeta^1$ is a line bundle.

Let $T$ be the tubular neighbourhood of $V$ in $M$ such that $f(T) \subset \wt T$.
The map $f\bigr|_T \co T \lra \wt T$ can be decomposed into a map $\wh f\co T \lra T'$ and the submersion $F\co T'\lra \wt T$, where $\wh f$ maps each fiber $D^2$ of the bundle $T \lra V$ into a fiber $D^3$ of $T' \lra V$ as a Whitney umbrella and $\wh f^{-1} (\partial T') = \partial T$.
On the boundary of each fiber $D^3$ we obtain a ``curved figure 8'' as image of~$\wh f$.
The manifold with boundary $M\setminus T$ will be denoted by $W$.
Note that its boundary is $\partial W = \partial_1 W \raisebox{-3pt}{$-$}\kern-11.43pt\shortparallel \partial_2 W$, where $\partial_1 W = \partial M$, and $\partial_2 W = \partial T$.

The image of $\partial_2W$ at $f$ is the union of the ($F$--images of the) above mentioned curved figures~8.
This is a codimension $2$ framed immersed submanifold in $D^{n + 1}$, we will denote it by $\wt V$.
(The first framing is the $F$--image of the inside normal vector of $\partial T'$ in $T'$.
The second framing is the normal vector of the curved figure~8 in $S^2 = \partial D^3$.)

It remained to show the following two claims.

\begin{claima}
\label{cl:a}
$\wt V$ with the given $2$--framing is framed cobordant to the immersion $\partial f\co \partial M \looparrowright S^n$ (compared with the framed embedding $S^n \subset D^{n + 1}$).
\end{claima}

\begin{claimb}
\label{cl:b}
$\wt V$ is obtained from the immersion $f\big| V\co V^{n - 2} \looparrowright D^{n + 1}$ by the figure 8 construction.
\end{claimb}

We have to make some remarks in order to clarify the above statements a) and b).

To a): The framed immersion $\partial f\co \partial M \looparrowright S^n$ and its composition with $i : S^n \subset D^{n + 1} \subset R^{n + 1}$ (with the added second framing, the inside normal vectors of $S^n$ in $D^{n + 1}$) represent the same element in $\pi^s(n - 1)$.
Indeed, the composition with $i$ corresponds to applying the suspension homomorphism in homotopy groups of spheres.
But the cobordism group of framed \emph{immersions} is isomorphic to the corresponding \emph{stable} homotopy group of spheres, so the suspension homomorphism gives the identity map of these groups.

To b): The figure 8 construction was defined for a codimension one immersed submanifold in a Euclidean space.
Here we apply it to the codimension~3 immersed submanifold $\wt V^{n - 2}$ in $D^{n + 1}$.
But $\wt V^{n - 2}$ has two linearly independent normal vector fields in $D^{n + 1}$ as was described above.
Identify $D^{n + 1}$ with $R^{n + 1}$ and apply the so-called multicompression theorem by Rourke--Sanderson \cite{RS}, thus one can make the two normal vectors parallel to the last two coordinate axes in $R^{n + 1}$, we can project the immersion to $R^{n - 1}$ and then we have a codimension $1$--immersion, so claim b) makes sense.
(This needs some more clarification since the multicompression theorem deals with embeddings. See below.)

\begin{proof}[Proof of Claim a)]
It can be supposed that the center $c$ of $D^{n + 1}$ does not belong to $\partial \wt T \cup f(M)$.
Let us omit from $D^{n + 1}$ a small ball centered around $c$ and still disjoint from $\partial \wt T \cup f(M)$.
In the remaining manifold $S^n \times I$ the direction of $I$ will be called vertical.
Take the product with an $R^q$ for big enough $q$, so that the immersions
$f\big|_W \co W^n \looparrowright S^n \times I$ and $\partial T \looparrowright S^n \times I$ become embeddings after small perturbations.

Now $W$ is embedded in $S^n \times I \times R^q$, it is framed with $q + 1$ normal vectors ($q$ are parallel to the coordinate axes of $R^q$).

One can suppose that the two boundary components of $W$ are embedded as follows

1) $\partial_1 W = \partial M$ is embedded into $S^n \times\{0\} \times R^q$.

2) $\partial_2 W = \partial T$ is embedded into the interior part $\text{\rm int}(S^n \times I) \times R^q$.

Both have $(q + 2)$-framings.

Now applying the multicompression theorem we make by an isotopy the first framing vector (the one coming from the normal vector of $\partial \wt T$ in $\wt T$) vertical, i.e.\ parallel to the direction of $I$ in $S^n \times I \times R^q$, while the $q$ last framing vectors (coming from $R^q$) we keep parallel to themselves.
The other boundary component $\partial_1 W \subset S^n \times \{0\} \times R^q$ is kept fixed.

We arrive at such an embedding of $W$ in $S^n \times I \times R^q$ for which the outward normal vector along $\partial_2 W$ in $W$ is vertical (i.e.\ parallel to the direction of $I$).
Now by a vertical shift we can deform $\partial_2 W$ into $S^n \times \{1\} \times R^q$.

This deformation can be extended to~$W$.
Projecting into $S^n \times I$ this new position of $W$ in $S^n \times I \times R^q$ we obtain an immersion cobordism between the immersions of the two boundary components.

On the first component $(\partial_1 W)$ we obtain $\partial [f]$.
On the second component we obtain the same framed cobordism class as was that of $\wt V$ in $\partial \wt T \subset D^{n + 1}$ (the union of curved figures 8).
Claim a) is proved.
\end{proof}

\begin{proof}[Proof of Claim b)]
Deform the immersed manifold $\wt V^{n - 1}$ (formed by the union of curved figures 8) as follows.
Contract each curved figure 8 by an isotopy along the corresponding sphere $S^2$ into a small neighbourhood of its double point obtaining an almost flat (very small) figure 8.
As we have noticed the normal bundle of $f(V)$ in $D^{n + 1}$ has the form $2\varepsilon ^1 \oplus \zeta^1 = \varepsilon_1^1 \oplus \varepsilon_2^1 \oplus \zeta^1$.
The first trivial normal line bundle $\varepsilon_1^1$ can be identified with the direction of the double line of the umbrella, the second one $\varepsilon_2^1$ can be the symmetry axes of the figures~8 (both of the curved and the flattened ones).
$\zeta^1$ is the direction orthogonal to the symmetry axes.

Now considering the maps of $V$ and $\wt V$ in $R^{n + 1}$ (instead of $D^{n + 1}$) applying again the multicompression theorem we make the two trivial normal directions $\varepsilon_2^1$ and $\varepsilon_1^1$ parallel to the last two coordinate axes and then project $V$ to $R^{n - 1}$.
In this way we obtain a new immersion $g\co V^{n - 2} \looparrowright R^{n - 1}$ and $\varepsilon_2^1$ (the symmetry axes of the figures~8) will be parallel to the normal vector of $R^{n - 1}$ in $R^n$.
Now the (flattened) figures~8 (of $\wt V$) are placed exactly as by the original figure~8 construction applied to~$g$.

It remained to note that the described deformations do not change the cobordism class of a framed immersion.
Claim~b) is proved.
\end{proof}

Thus Theorem~\ref{th:A} is also proved.
\end{proof}

\begin{rem}
The stable homotopy groups of $RP^\infty$ were computed by Liulevicius \cite{Liu} in dimensions not greater than~$9$.
\end{rem}

Below in the first line we show his result, in the second one the stable homotopy groups of spheres.
These two lines by Theorem~\ref{th:A} give the groups $\Cob\, \Sigma^{1,0}(n)$ for $n \leq 10$ given in the third line. (Here for example $(Z_2)^3$ stands for $Z_2 \oplus Z_2 \oplus Z_2$.)

\begin{tabular}{@{\hspace*{8 pt}}c@{\hspace*{7.5 pt}}|@{\hspace*{7.5 pt}}c@{\hspace*{7.5 pt}}|@{\hspace*{7.5 pt}}c@{\hspace*{7.5 pt}}|@{\hspace*{7.5 pt}}c@{\hspace*{7.5 pt}}|@{\hspace*{7.5 pt}}c@{\hspace*{7.5 pt}}|@{\hspace*{7.5 pt}}c@{\hspace*{7.5 pt}}|@{\hspace*{7.5 pt}}c@{\hspace*{7.5 pt}}|@{\hspace*{7.5 pt}}c@{\hspace*{7.5 pt}}|@{\hspace*{7.5 pt}}c@{\hspace*{7.5 pt}}|@{\hspace*{7.5 pt}}c@{\hspace*{8 pt}}|@{\hspace*{7.5 pt}}c@{\hspace*{7.5 pt}}}
$n$ & $1$ & $2$ & $3$ & $4$ & $5$ & $6$ & $7$ & $8$ & $9$ & $10$ \\
\hline
$\pi_n^s(RP^\infty)$ & $Z_2$ & $Z_2$ & $Z_8$ & $Z_2$ & $0$ & $Z_2$ & $Z_{16} \oplus Z_2$ &
$(Z_2)^3$ & $(Z_2)^4$ & ? \\
\hline
$\pi^s(n)$ & $Z_2$ & $Z_2$ & $Z_{24}$ & $0$ & $0$ & $Z_2$ & $Z_{240}$ & $(Z_2)^2$ & $(Z_2)^3$ & $Z_6$ \\
\hline
$\Cob\, \Sigma^{1,0}(n)$ & $0$ & $0$ & $Z_3$ & $Z_2$ & $Z_2$ & $0$ & $Z_{15}$ & $Z_2$ & $Z_2$ & $Z_6$ \\
\hline
\end{tabular}

\part{Cusp maps}
\label{part:II}

\section{Formulation of the results}
\label{sec:II1}

Here we consider cusp maps, i.e.\ maps having at most cusp singularities.
(In the previous terms these are $\tau$--maps for $\tau = \{\Sigma^0, \Sigma^{1,0}, \Sigma^{1,1}\}$.)
The cobordism group of cusp maps of oriented $n$--dimensional manifolds in $R^{n + 1}$ will be denoted by $\Cob\, \Sigma^{1,1}(n)$.
We shall compute these groups modulo their $2$--primary and $3$--primary parts.
Let $C_{\{2, 3\}}$ be the minimal class of groups containing all $2$--primary and $3$--primary groups.

\renewcommand{\thethm}{\Alph{thm}}
\setcounter{thm}{1}

\begin{thm}
\label{th:B}
\hspace*{60pt}$\Cob\, \Sigma^{1,1}(n) \underset{\mathcal C_{\{2,3\}}}{\approx} \pi^s(n) \oplus \pi^s(n - 4)$\\
where $ \underset{\mathcal C_{\{2,3\}}}{\approx}$ means isomorphism modulo the class $\mathcal C_{\{2,3\}}$, and $\pi^s(m)$ denotes the $m$\textsuperscript{th}  stable homotopy group of spheres.
\end{thm}

\renewcommand{\thethm}{\arabic{thm}}
\setcounter{thm}{3}

\section{Preliminaries on Morin maps}
\label{sec:II2}

Morin maps are those of types $\Sigma^{1,0}, \Sigma^{1,1,0}, \dots, \Sigma^{1_r, 0}, \dots$, $r = 1,2, \dots$ (See \cite{AGV}.)

For $\eta = \Sigma^{1_r,0}$ the universal normal bundle $\wt\xi_\eta$ will be denoted by $\wt\xi_r$.
It was established in \cite{RSz} and \cite{R} that the structure group of $\wt\xi_r$ is $Z_2$ and the bundle $\wt\xi_r$ is associated to a representation $\lambda_2\co Z_2 \lra O(2r + 1)$ with the property that $\lambda_2(Z_2) \subset SO(2r + 1)$ precisely when $r$ is even.
It follows that $\wt \xi_r$ is the direct sum $i \cdot\gamma^1 \oplus j \cdot \varepsilon^1$, where $i + j = 2r + 1$, and $i\equiv r \text{\rm mod }2$.
Here $\varepsilon^1$, $\gamma^1$ are the trivial and the universal line bundles respectively.
Hence the Thom space $T\wt\xi_r$ is $S^j(RP^\infty / RP^{i - 1})$.
It is easy to see that for any odd $p$ the reduced $\text{\rm mod }p$ cohomology $\overline H^*$ $(T\wt\xi_r; Z_p)$ vanishes if $r$ is odd, and the natural inclusion $S^{2r + 1} \subset T\wt\xi_r$ (as a ``fiber'') induces isomorphism of the cohomology groups with $Z_p$--coefficients for $r$ even.
Consequently by Serre's generalization of the Whitehead theorem \cite{S2} -- the inclusion $\Gamma S^{2r + 1} \subset \Gamma T\wt\xi_r$ (recall $\Gamma = \Omega^\infty S^\infty$) induces isomorphism  of the odd torsion parts of the homotopy groups for $r$ even, while for $r$ odd $\pi_*(\Gamma T \wt\xi_r)$ are finite $2$--primary groups.

\section{Computation of the cusp cobordism groups}
\label{sec:II3}

In the case of cusps $r = 2$ and we have that the inclusion $\Gamma S^5 \subset \Gamma T \wt\xi_2$ is a $\text{\rm mod }\mathcal C_2$ homotopy equivalence ($\mathcal C_2$ is the class of $2$--primary groups).

Let us consider the following pull-back diagram defining the space $X^{\fr} \Sigma^{1,1}$
\[
\begin{CD}
X^{\fr} \Sigma^{1,1} @>>> X\Sigma^{1,1} \\
@VV {\displaystyle{X\Sigma^{1,0}}} V @VV {\displaystyle{X\Sigma^{1,0}}}V\\
\Gamma S^5 @>>> \Gamma T\wt\xi_2
\end{CD}
\]
(Note that $X^{\fr} \Sigma^{1,1}$ is the classifying space of those cusp maps for which the normal bundle of the $\Sigma^{1,1}$--stratum in the target is trivialized.
Equivalently these are the cusp maps for which the kernel of the differential is trivialized over the cusp-stratum.)

The horizontal maps of the diagram induce isomorphisms of the odd-torsion parts of the homotopy groups.
Now we show that the homotopy exact sequence of the left-hand side fibration ``almost has a splitting''.

\begin{defn}
Let $p\co E\ \overset F{\lra}\ B$ be a fibration and let $t$ be a natural number.
We say that this fibration has an algebraic $t$--splitting if for each $i$ there is homomorphism $S_i\co \pi_i(B) \lra \pi_i(E)$ such that the composition of $S_i$ with the map $p_*$ induced by $p$ is a multiplication by~$t$.
We say that the fibration $p$ has a geometric $t$--splitting if it has an algebraic one such that all $S_i$ are induced by a map $s\co B \lra E$ (the same map $s$ for each $i$).
\end{defn}

\begin{lem}
\label{lem:4}
The fibration $X^{\fr} \Sigma^{1,1} \lra \Gamma S^5$ has a $6$--splitting.
\end{lem}

\begin{rem}
We shall prove this only in algebraic sense, since we will need only that.
For the existence of geometric splitting we give only a hint.
\end{rem}

\begin{proof}[Proof of Theorem~\ref{th:B}]
is immediate from Lemma~\ref{lem:4}.
Indeed,
\[
\Cob\, \Sigma^{1,1}(n) \approx \pi_{n + 1}(X \Sigma^{1,1})\underset{\mathcal C_2}{\approx} \pi_{n + 1} (X^{\fr} \Sigma^{1,1}).
\]
Now the homotopy exact sequence of the fibration
$p^{\fr} \co X^{\fr} \Sigma^{1,1}\ \overset{X\Sigma^{1,0}}{\hbox to12mm{\rightarrowfill}}\ \Gamma S^5$ has a $6$--splitting, hence modulo the class $\mathcal C_{\{2,3\}}$ we have
\[
\pi_{n + 1}(X^{\fr} \Sigma^{1,1}) \underset{\mathcal C_{\{2,3\}}}{\approx} \pi_{n + 1}^s(S^5) \oplus
\pi_{n + 1} (X\Sigma^{1,0}) \underset{\mathcal C_2}{\approx} \pi^s(n - 4) \oplus \pi^s(n).
\]
(In the last $\text{\rm mod }\mathcal C_2$ isomorphism we used Theorem~\ref{th:A}.)

Theorem~\ref{th:B} is proved except Lemma~\ref{lem:4}.
\end{proof}

\begin{proof}[Proof of Lemma~\ref{lem:4}] will follow from the following two claims.

\begin{claimnumbered}
\label{claim:1}
If there is a map of an oriented $4$--dimensional manifold into $R^5$ with $t$ cusp points (algebraically counting them), then the fibration $p^{\fr}$ has a $t$--splitting (algebraically).
\end{claimnumbered}

\begin{claimnumbered}
\label{claim:2}
There is a cusp-map $f\co M^4 \lra R^5$ with $t$--cusp points (counting algebraically).
\end{claimnumbered}

\begin{proof}[Proof of Claim~\ref{claim:1}]
Let $f\co  M^4 \lra R^5$ be a cusp-map with $t$ cusp-points.
Let $x$ be an element in $\pi_m(\Gamma S^5) \approx \pi^s(m - 5)$.
It can be represented by a framed, immersed $(m - 5)$--dimensional manifold $A^{m - 5}$ in $R^m$, let us denote its immersion by $\alpha$.

Take the product $A^{m - 5} \times M^4$ and its map into the direct product $A^{m - 5} \times D^5$ by $\text{\rm id}_A \times f$.
Now the target $A^{m - 5} \times D^5$ can be mapped by a submersion $F$ into $R^m$ onto the immersed tubular neighbourhood of $\alpha(A)$ using the framing to map the $D^5$--fibers).
The composition $A \times M^4 \lra R^m$ is clearly a cusp map and its cusp-singularity stratum represents the element $t \cdot x$ in $\pi^s(m - 5)$.

Claim~\ref{claim:1} is proved (at least its algebraic version.
The geometric one follows from the fact that we use  the same element $[f\co M^4 \lra D^5]$ for any element $x \in \pi^s(i)$ and for any dimension $i$ to construct the element $S_i(x)$.
The classifying space $\Gamma S^5$ can be obtained as the limit of target spaces of codimension 5 framed immersions.)
\end{proof}

\begin{proof}[Proof of Claim~\ref{claim:2}]
is a compilation of the following two theorems.

\begin{theor}[\cite{Sz1}, \cite{Sz2}, \cite{L}]
Given a generic immersion $g\co M \looparrowright Q \times R^1$ and a natural number $r$, let us denote by $\Delta_{r + 1}(g)$ the manifold of (at least) $(r + 1)$--tuple points in $M^n$.
Let $f\co M \lra Q$ be the composition of $g$ with the projection $Q \times R^1 \lra Q$.
Let us denote by $\Sigma^{1_r}(f)$ the closure of the set of $\Sigma^{1_r}$ singular points of~$f$.
Then the manifolds $\Delta_{r + 1}(g)$ and $\Sigma^{1_r}(f)$ are cobordant.
If $M$ and $Q$ are oriented and $\text{\rm dim }Q - \text{\rm dim }M$ is odd, then these manifolds are oriented and they are oriented-cobordant.
\end{theor}

\begin{theor}[Eccles--Mitchell \cite{EM}]
There is an oriented closed $4$--dimensional manifold $M^4$ and an immersion $g\co M^4 \looparrowright R^6$ with (algebraically) $2$ triple points.
\end{theor}
\end{proof}
Theorem~\ref{th:B} is proved.
\end{proof}

\goodbreak

\part{Higher Morin maps}
\label{part:III}

Most of the previous arguments can be applied in the computation of cobordism groups of Morin maps having at most $\Sigma^{1_r}$ singular points for any $r$ (the codimensions of the considered maps are still equal to one, and the maps are cooriented).
The only problem is that we need a generalization of the Theorem of Eccles--Mitchell.

Below we shall give a weak form of such a generalization.
This will allow us to compute the groups $\Cob\, \Sigma^{1_r}(n)$ modulo the $p$--primary part for $p \leq r + 1$.

Notation: Let $\mathcal C\{ p \leq 2r + 1\}$ denote the minimal class of groups containing all $p$--primary groups for any prime $p \leq 2 r + 1$.
The main result of this Part~\ref{part:III} is the following.

\renewcommand{\thethm}{\Alph{thm}}
\setcounter{thm}{2}

\begin{thm}
\label{th:C}
Let us denote by $\Cob \Sigma^{1_i}(n)$ the cobordism group of $\Sigma^{1_i}$--map of oriented $n$--manifolds in $R^{n + 1}$ (i.e.\ $\tau$--maps for $\tau = \{\Sigma^0, \Sigma^{1,0},\dots, \Sigma^{1,1,\dots, 1}, \ i \text{ digits }1\}$).
Then for any $r$
\[
\Cob\, \Sigma^{1_{2r + 1}}(n) \underset{\mathcal C_2}{\approx} \Cob\, \Sigma^{1_{2r}}(n) \underset{\mathcal C\{p \leq 2r + 1\}}{\approx} \bigoplus\limits_{i = 0}^r \pi^s(n - 4i).
\]
\end{thm}

The proof is very similar to that given in Parts~\ref{part:I} and \ref{part:II}.
It goes by induction on~$r$.
First we give a weak analogue of the Theorem of Eccles and Mitchell.

\renewcommand{\thethm}{\arabic{thm}}
\setcounter{thm}{0}
\begin{lem}
\label{lem:1}
\emph{\rm a)}
For any natural number $k$ there is a positive integer $t(k)$ such that for any immersion of an oriented, closed, smooth $4k$--dimensional manifold in $R^{4k + 2}$ the algebraic number of $(2k + 1)$--tuple points is divisible by $t(k)$, and there is a case when this number is precisely $t(k)$.

{\rm b)} The number $t(k)$ coincides with the order of the cokernel of the stable Hurewicz homomorphism
\[
\pi_{4k + 2}^s (CP^\infty) \lra H_{4k + 2} (CP^\infty).
\]
\end{lem}

Before proving Lemma~\ref{lem:1} it will be useful to recall a result on the cokernel of the stable Hurewicz map.

\begin{theor}[Arlettaz \cite{A}]
Let $X$ be a $(b - 1)$--connected space and let $\varrho_j$ be the exponent of the stable homotopy group of spheres $\pi^s(j)$.
Let $h_m\co \pi_m^s(X) \lra H_m(X)$ be the stable Hurewicz homomorphism.
Then $(\varrho_1 \dots \varrho_{m - b - 1})(\text{\rm coker } h_m) = 0$.
\end{theor}

Next we recall a theorem of Serre on the prime divisors of the numbers $\varrho_j$.

\begin{theor}[Serre \cite{S1}]
$\pi^s(i) \otimes Z_p = 0$ if $i < 2p - 3$ and $\pi^s(2p - 3) \otimes Z_p = Z_p$.
\end{theor}

Hence $\varrho_j$ is not divisible by a prime $p$ if $p > \frac{j + 3}{2}$,
in other words, $\varrho_j$ may have a prime $p$ as a divisor only if $p \leq \frac{j + 3}{2}$.

Applying Arlettaz' theorem to $X = CP^\infty$, $b = 2$, $m = 4r + 2$ we obtain that $t(r)$ has no prime divisor greater than $2r + 1$.

\begin{proof}[Proof of Theorem~\ref{th:C}]
should be clear now, since it is completely analogous to that of Theorem~\ref{th:B}.

First we consider the ``key bundle''
\[
X\Sigma^{1_{2r + 1}} \lra \Gamma T \wt\xi_{2r + 1} \ \text{ with fiber } \ X \Sigma^{1_{2r}}.
\]
Remember that $H^*(\Gamma T \wt\xi_{2r + 1}; Z_p) = 0$ for any odd $p$, so -- by the $\text{\rm mod }\mathcal C$ Whitehead theorem \cite{S2} we obtain the first ($\text{\rm mod } \mathcal C_2$) isomorphism in the Theorem
\[
\Cob\, \Sigma^{1_{2r + 1}} (n) \underset{\mathcal C_2}{\approx} \Cob\, \Sigma^{1_{2r}}(n).
\]
In order to prove the second ($\text{\rm mod }\mathcal C\{p \leq 2 r + 1\}$) isomorphism recall that modulo the class $\mathcal C_2$ the key bundle
\[
X \Sigma^{1_{2r}} \lra \Gamma T\wt\xi_{2r} \ \ \text{ with fibre }\ X \Sigma^{1_{2r - 1}}
\]
can be replaced by the bundle
\[
X^{\fr} \Sigma^{1_{2r}} \lra \Gamma S^{4r + 1} \ \text{ with the same fibre.}
\]
The later bundle has a(n algebraic) $t(r)$--splitting.

By Lemma~\ref{lem:1} and the theorems of Arlettaz and Serre $t(r)$ has no prime divisor greater than $2r + 1$, hence by induction on~$r$ we obtain the second isomorphism in Theorem~\ref{th:C}.
\end{proof}

\begin{proof}[Proof of Lemma~\ref{lem:1}]
By Herbert's theorem the algebraic number of the $(2k + 1)$--tuple points of an immersion $f\co M^{4k} \looparrowright R^{4k + 2}$ is $\left<{\overline p_1}^k , [M^{4k}]\right>$.
The immersion $f$represents an element $[f]$ of the corresponding cobordism group of immersions of oriented $4k$--manifolds in $R^{4k + 2}$.
This cobordism group is isomorphic to the group $\pi_{4k + 2}^s (MSO(2))$, the element of the later group corresponding to $[f]$ will be denoted by $[\alpha_f]$.
Here $\alpha_f$ is the Pontrjagin--Thom map $S^{q + 4k + 2} \lra S^q MSO(2)$, for $q$ big enough.

Let us consider the following composition of maps
\[
\pi_{4k + 2}^s(MSO(2))\ {\underset \circledone{\lra}}\ H_{4k + 2 + q} \bigl(S^q MSO(2))
\ \overset{\displaystyle\approx}{\underset\circledtwo{\lra}}\ H_{4k} (BSO(2))
\ \overset{\displaystyle\approx}{\underset\circledthree{\lra}}\ Z.
\]
Here $\circledone$ is the stable Hurewicz homomorphism,
$\circledtwo$ is the Thom isomorphism in the homologies
\[
x \lra S^q U_2 \cap x
\]
where $U_2$ is the Thom class of $MSO(2)$ and $S^q U_2$ its $q$\textsuperscript{th}  suspension, and also the Thom class of $S^q MSO(2)$.

$\circledthree$ is the evaluation on the class ${p_1}^k$
\[
y \lra \langle y, {p_1}^k\rangle.
\]

Since the maps $\circledtwo$ and $\circledthree$ are isomorphisms, the cokernel of this composition is the same as the cokernel of $\circledone$, i.e.\ of the stable Hurewicz homomorphism.

On the other hand, we show that the image of this composition map is $t(k)  Z$, and that will prove
part~b) of Lemma~\ref{lem:1}.
(Part~a) follows then as well, since the rational stable Hurewicz homomorphism
\[
\pi_m^s(X) \otimes Q \lra H_m(X; Q) \ \text{ is an isomorphism.})
\]

\begin{claim}
The composition of the maps $\circledone,\ \circledtwo,\ \circledthree$ has image $t(k) \cdot Z$.
\end{claim}

%%\begin{proof}
{\bf Proof}\quad
It is enough to show that the image of $[\alpha_f] \in \pi_{4k + 2}^s (MSO(2))$ is $\left< {\overline p_1}^{k}, [M^{4k}]\right>$.

$[\alpha_f]$ goes by the map $\circledone$ to $(\alpha_f)_*[S^{q + 4k + 2}]$, that is mapped by $\circledtwo$ to $(\alpha_f)_*[S^{q + 4k + 2}]\cap S^q U_2$.

Let $\nu$ be the normal bundle of $f$, let us denote by $T\nu$ its Thom space, let $pr\co S^{q + 4k + 2} \lra S^q T\nu$ be the Pontrjagin--Thom map, $\beta_f\co S^q T\nu \lra S^q\, MSO(2)$ the fiberwise map of Thom spaces that on the base spaces is the map $\nu_f\co M \lra BSO(2)$ inducing the normal bundle~$\nu$.

Now $(\alpha_f)_*[S^{q + 4k + 2}] = (\beta_f)_* \circ pr_*[S^{q + 4k + 2}] = (\beta_f)_* [S^q T\nu]$.
Here $[S^q T\nu]$ is the fundamental homology class of $S^q T\nu$.
Therefore
\[
\aligned
\left<{p_1}^k\!, (\alpha_f)_* [S^{q + 4k + 2}] \cap S^q U_2\right>
&= \left<{p_1}^k\!, (\beta_f)_* [S^q T\nu] \cap S^q U_2\right> \\
&= \left<{p_1}^k\!, (\nu_f)_* [M] \right>
= \left< {\nu_f}^*{p_1}^k\!, [M]\right>
= \left< {\overline p_1}^k\!, [M]\right>\!. \ \text{ Q.E.D.}
\endaligned
\]
\end{proof}

%%\end{proof}

\end{document}